\documentclass[11pt]{amsart}    

\usepackage[pagebackref=true,colorlinks,linkcolor=blue,citecolor=magenta]{hyperref}
\usepackage{amsmath,amssymb,latexsym} 
\usepackage{graphicx}\usepackage{centernot}

\usepackage{fullpage}
\usepackage[square, numbers]{natbib}   
\usepackage{comment}

\usepackage{enumitem}

\newtheorem{theorem}{Theorem} 

\newtheorem{alphtheorem}{Theorem}

\newtheorem{alphlemma}{Lemma}

\newtheorem{corollary}{Corollary}

\newtheorem{lemma}{Lemma}
\theoremstyle{definition}

\theoremstyle{remark}

\def\Z{\mathbb{Z}}
\def\T{\mathsf{T}}

\def\zero{\boldsymbol{0}}

\def\ds{\displaystyle}
\def\KG{\operatorname{KG}}
\def\SG{\operatorname{SG}}
\def\cd{\operatorname{cd}}

\def\alt{\operatorname{alt}}
\def\salt{\operatorname{salt}}
\def\HH{\mathcal{H}}
\def\T{\mathcal{T}}

\def\mod{\operatorname{mod}}

\title{Chromatic Number of Random Kneser Hypergraphs}
\author{Meysam Alishahi}
\address{M. Alishahi, 
School of Mathematical Sciences,
Shahrood University of Technology, Shahrood, Iran}
\email{meysam\_alishahi@shahroodut.ac.ir}

\author{Hossein Hajiabolhassan}
\address{H. Hajiabolhassan,
Department of Mathematical Sciences,
Shahid Beheshti University, P.O. Box 19839-69411, Tehran, Iran \newline
School of Mathematics, Institute for Research in Fundamental Sciences (IPM), P.O. Box 19395-5746, Tehran, Iran}
\email{hhaji@sbu.ac.ir}

\begin{document}
\maketitle

\begin{abstract} 
Recently, Kupavskii~[{\it On random subgraphs of {K}neser and {S}chrijver 
graphs.  J. Combin. Theory Ser. A, {\rm 2016}.}] investigated the chromatic 
number of random Kneser graphs 
$\KG_{n,k}(\rho)$ and proved that, in many cases, 
the chromatic numbers of  the random Kneser graph $\KG_{n,k}(\rho)$ 
and  the Kneser graph $\KG_{n,k}$ are almost surely closed. He also marked the studying of the chromatic 
number of random Kneser hypergraphs $\KG^r_{n,k}(\rho)$ as a very interesting problem. 
With the help of $\Z_p$-Tucker lemma, a combinatorial generalization of the Borsuk-Ulam theorem, we generalize Kupavskii's result to random general  Kneser hypergraphs by   
introducing an almost surely lower bound for the chromatic number of them. 
Roughly speaking, as a special case of our result,  we show that the chromatic numbers of  the random Kneser hypergraph $\KG^r_{n,k}(\rho)$ and the Kneser hypergraph $\KG^r_{n,k}$ are almost surely closed in many cases.  Moreover,  restricting to the Kneser and {S}chrijver graphs, we present a purely combinatorial proof for  an improvement of Kupavskii's results. 

Also, for any hypergraph $\HH$, we present a lower bound for  
the minimum number of colors required in 
a coloring of $\KG^r(\mathcal{H})$ with no monochromatic  $K_{t,\ldots,t}^r$ subhypergraph, 
where $K_{t,\ldots,t}^r$ is the complete $r$-uniform $r$-partite hypergraph with $t r$ 
vertices such that each of its parts has $t$ vertices. 
This result generalizes the lower bound for the chromatic number of 
$\KG^r(\mathcal{H})$ found by the present authors~[{\it On the chromatic number of general 
{K}neser hypergraphs.  J.  Combin. Theory, Ser. B, 
  {\rm 2015}.}]. \\

\end{abstract}
\section{Introduction and Main Results}
For positive integers $n$ and $k$, by the symbols $[n]$ and ${[n]\choose k}$, we mean 
the set $\{1,\ldots,n\}$ and the set of all $k$-subsets of $[n]$, respectively.  
A hypergraph $\mathcal{H}$ is a pair $(V(\mathcal{H}),E(\mathcal{H}))$, where 
$V(\mathcal{H})$ is a finite nonempty set and $E(\mathcal{H})$ is a family of distinct
nonempty subsets of $V(\mathcal{H})$. Respectively, the sets $V(\mathcal{H})$ and $E(\mathcal{H})$ are called the vertex set and the edge set of $\mathcal{H}$.
If each  edge of $\mathcal{H}$ has the cardinality $r$, then $\mathcal{H}$ is called $r$-uniform. A $2$-uniform hypergraph is simply called a graph. 
Let $\mathcal{H}$ be  an $r$-uniform hypergraph and $V_1,\ldots,V_r$ be  pairwise disjoint subsets of $V(\mathcal{H})$. 
The hypergraph $\mathcal{H}[V_1,\ldots,V_r]$ is a subhypergraph of $\mathcal{H}$ 
whose vertex set and edge set are respectively   
$\ds\bigcup_{i=1}^rV_i$ and $$E(\mathcal{H}[U_1,\ldots, U_r])=\left\{e\in E(\mathcal{H}):\; e\subseteq \ds\bigcup_{i=1}^r U_i\mbox{ and } |e\cap U_i|= 1\mbox{ for each } i\in[r]\right\}.$$

For a positive integer $r\geq 2$,  the Kneser hypergraph $\KG^r_{n,k}$ is a 
hypergraph which has the vertex set ${[n]\choose k}$, and whose edges are   
formed by the $r$-sets $\{e_1,\ldots,e_r\}$, where $e_1,\ldots,e_r$ are pairwise disjoint members of ${[n]\choose k}$. 
Kneser 1955~\cite{MR0068536} conjectured that  for $n\geq 2k$, the chromatic number of $\KG^2_{n,k}$ is
$n-2k+2$. After more than 20 years, in a  fascinating paper, 
Lov{\'a}sz~\cite{MR514625}  gave an affirmative answer to Kneser's conjecture 
using algebraic topology. Lov{\'a}sz's paper is known as the beginning of the 
study of combinatorial problems by using topological tools, which is called topological combinatorics. Later, in~1986, Alon, Frankl and 
Lov\'asz~\cite{MR857448} generalized Lov{\'a}sz's result to Kneser hypergraphs by proving that for $n\geq rk$, 
$$\chi(\KG^r_{n,k})=\left\lceil{n-r(k-1)\over r-1}\right\rceil.$$ 
This result also gives a positive answer to a conjecture posed by Erd\H{o}s~\cite{MR0465878}. 
Schrijver~\cite{MR512648} improved Lov{\'a}sz's result 
by introducing a subgraph $\SG_{n,k}$ of $\KG^2_{n,k}$, called the Schrijver graph,  
which is a vertex critical graph having the same chromatic number as that of $\KG^2_{n,k}$.
A {\it stable subset of $[n]$} is a set $A\subseteq [n]$ such that for 
each $i\neq j\in A$, we have
$2\leq |i-j|\leq n-2$.
Let ${[n]\choose k}_{stable}$ be the set of all stable $k$-subsets of $[n]$.  
The graph $\SG_{n,k}=\KG\left([n],{[n]\choose k}_{stable}\right)$ is called the Schrijver graph.

For a hypergraph $\mathcal{H}$ and a positive integer $r\geq 2$, the general Kneser 
hypergraph $\KG^r(\mathcal{H})$ is an $r$-uniform hypergraph with vertex 
set $E(\mathcal{H})$ 
and the edge set defining as follows;
$$E(\KG^r(\mathcal{H}))=\left\{\{e_1,\ldots,e_r\}\subseteq E(\mathcal{H}):\; 
e_i\cap e_j=\varnothing\mbox{ for each } i\neq j\in[r]\right\}.$$ 
Throughout the paper, for $r=2$, we speak about $\KG(\mathcal{H})$ and $\KG_{n,k}$
rather than $\KG^2(\mathcal{H})$ and $\KG^2_{n,k}$, respectively. The $r$-colorability defect of ${\mathcal H}$, denoted ${\rm cd}_r({\mathcal H})$, is the 
minimum number of vertices should be excluded so that the induced subhypergraph on the remaining vertices is $r$-colorable.
Note that if we set $K_n^k=([n],{[n]\choose k})$, then $\KG^r(K_n^k)=\KG^r_{n,k}$ 
and $\cd_r(K_n^k)=n-r(k-1)$ for $n\geq rk$. 
Dol'nikov~\cite{MR953021}~(for $r=2$) and K{\v{r}}{\'{\i}}{\v{z}}~{\rm \cite{MR1081939} }improved  the results by Lov\'asz~\cite{MR514625} and  Alon, Frankl and 
Lov\'asz~\cite{MR857448} by proving 
$\chi(\KG^r(\mathcal{H}))\geq\left\lceil{\cd_r(\mathcal{H})\over r-1}\right\rceil.$ 
A famous combinatorial counterpart of the Borsuk-Ulam theorem is Tucker lemma~\cite{MR0020254}.
Matou{\v{s}}ek~\cite{MR2057690} proved Lov\'asz's theorem by use of Tucker lemma.
He also presented a purely combinatorial proof for 
Tucker lemma, hence a purely combinatorial proof 
for Lov\'asz's theorem. 
Ziegler~\cite{MR1893009} extended Tucker lemma to $\Z_p$-Tucker lemma with a proof which makes no use of topological tools.
Using this lemma, Ziegler~\cite{MR1893009}, inspired by Matou{\v{s}}ek's proof,  improved 
Dol'nikov-K{\v{r}}{\'{\i}}{\v{z}} lower bound by a purely combinatorial approach. 
Next, Meunier~\cite{MR2793613} found a variant of  $\Z_p$-Tucker lemma as an 
extension of Ziegler's result, which can be proved combinatorially as well. Using this lemma, he presented a combinatorial proof of Schrijver's result.\\

\noindent{\bf Remark.} Note that since there is a purely combinatorial proof for $\Z_p$-Tucker 
lemma (Lemma~\ref{zptucker}), see~\cite{MR2793613,MR1893009}, 
any combinatorial proof with the help 
of $\Z_p$-Tucker lemma can be seen as a purely combinatorial proof. 
In this point of view, all results in this paper are proved purely combinatorial. \\

Let $\mathbb{Z}_r=\{\omega^1,\ldots,\omega^r\}$ be a cyclic group with generator $\omega$.
For an $X=(x_1,\ldots,x_n)\in(\mathbb{Z}_r\cup\{0\})^n$, an {\it alternating 
subsequence of $X$} is a sequence $x_{i_1},x_{i_2},\ldots,x_{i_k}$ ($i_1<\cdots <i_k$) of nonzero terms of $X$ such that $x_{i_j}\neq x_{i_{j+1}}$ for each $j\in[k-1]$.
The maximum length of an alternating subsequence of $X$ is called 
{\it the alternation number of $X$}, denoted ${\rm alt}(X)$. We define $\alt(0,\ldots,0)=0$.
For each $i\in[r]$, let $X^i$ be the set of all $j\in[n]$
such that $x_j=\omega^i$, that is, $X^i=\{j\in[n]\;:\; x_j=\omega^i\}$. 
Note that, by abuse of notation, we can write $X=(X^1,\ldots,X^r)$.
For two signed vectors $X$ and $Y$, by $X\subseteq Y$, we mean
$X^i\subseteq Y^i$ for each $i\in [r]$.
Let $\mathcal{H}$ be a hypergraph and let 
$\sigma: [n]\longrightarrow V(\mathcal{H})$ be a bijection.
Define 
$${\rm alt}_r(\mathcal{H},\sigma,q)=\max\left\{{\rm alt}(X):\; X\in(\Z_r\cup\{0\})^n\;\mbox{ s.t. } |E(\mathcal{H}[\sigma(X^i)])|\leq q-1\mbox{ for all } i\in[r]\;\right\}.$$
Now, set 
$${\rm alt}_r(\mathcal{H},q)=\min_\sigma {\rm alt}_r(\mathcal{H},\sigma,q),$$
where the minimum is taken over all bijections $\sigma:[n]\longrightarrow V(\mathcal{H}).$
Throughout the paper, for $q=1$, 
we would use ${\rm alt}_r(\mathcal{H})$ rather than ${\rm alt}_r(\mathcal{H},1)$.
The present authors~\cite{2013arXiv1302.5394A}, using the extension of 
$\Z_p$-Tucker lemma by 
Meunier~\cite{MR2793613}, improved Dol'nikov-K{\v{r}}{\'{\i}}{\v{z}} lower 
bound by proving that
\begin{equation}\label{alihajijctb} 
\chi(\KG^r(\mathcal{H}))\geq
\left\lceil{|V(\mathcal{H})|-\alt_r(\mathcal{H})\over r-1}\right\rceil.
\end{equation}
Using this lower bound, the chromatic numbers of several families 
of graphs  and hypergraphs are computed, see~\cite{2013arXiv1302.5394A,2014arXiv1401.0138A,2014arXiv1403.4404A,2014arXiv1407.8035A,2015arXiv150708456A,HaMe16}.\\

\subsection{\bf Random Kneser Hypergraphs.}
Let $\rho$ be a real number, where $0<\rho\leq 1$. The random general Kneser hypergraph 
$\KG^r(\mathcal{H})(\rho)$ is a random spanning subgraph of $\KG^r(\mathcal{H})$ containing each edge 
of $\KG^r(\mathcal{H})$ randomly and independently with probability $\rho$, i.e., each pairwise vertex-disjoint edges $e_1,\ldots,e_r\in E(\mathcal{H})$ 
form an edge of $\KG^r(\mathcal{H})(\rho)$ with probability $\rho$. 
The stability properties of random Kneser graphs $\KG_{n,k}(\rho)$  
has been received a considerable 
attention in recent years, see for instances~\cite{MR3482268,MR3403515}. 
In this regard, Bollob{\'a}s, Narayanan,  and Raigorodskii~\cite{MR3403515}
proved a random analogue of the Erd\H{o}s-Ko-Rado theorem. In detail, 
they proved that 
for a real number $\varepsilon>0$ and an integer function $2\leq k=k(n)=o(n^{1\over 2})$, 
there is a threshold $t(n)\in (0,1]$ such that for $\rho\geq (1+\varepsilon)t(n)$
and $\rho\leq (1-\varepsilon)t(n)$, the quantity $\ds{\rm Pr}\left(\alpha\left(\KG_{n,k}(\rho)\right)={n-1\choose k-1}\right)$ respectively tends to $1$ and $0$ as $n$ goes to infinity. 
They also asked what happens for larger $k$. Furthermore, they conjectured that if
${k\over n}$ is bounded away from ${1\over 2}$, then such a random analogue of the Erd\H{o}s-Ko-Rado theorem should continue to hold for some $\rho$ bounded away from $1$.
This conjecture received an affirmative answer owing to the work by 
Balogh, Bollob{\'a}s, and Narayanan~\cite{MR3482268}. They proved that the 
random analogue of the Erd\H{o}s-Ko-Rado theorem is still true for 
each $k\leq ({1\over 2}-\varepsilon)n$.

In the rest of the paper, for simplicity of notation, for two functions $f(n)$ and $g(n)$, 
by $f(n)\gg g(n)$ or $g(n)\ll f(n)$, we mean $\ds\lim_{n\mapsto\infty} {g(n)\over f(n)}=0$.
Also, the abbreviation a.s. stands for ``almost surely'', which means that the probability tends to $1$ as $n$ goes to infinity.

Recently, Kupavskii~\cite{MR3479235} studied 
the chromatic number of random Kneser graphs $\KG_{n,k}(\rho)$. 
He applied Gale's lemma~\cite{MR0085552},  in a similar fashion as in B{\'a}r{\'a}ny's proof~\cite{MR514626} of Lov\'asz's theorem, 
to introduce an a.s. lower bound for the chromatic number of random Kneser graphs $\KG_{n,k}(\rho)$. 
The following theorem is the main result of Kupavskii's paper.
\begin{alphtheorem}{\rm\cite{MR3479235}}\label{Kupavskiimain}
Let $k=k(n)\geq 2$ and $l=l(n)\geq 1$ be integer functions and $\rho=\rho(n)\in(0,1]$ 
be a real function such that  $d=n-2k-2l+2\geq 3$. Put $x=\ds\left\lceil{{k+l\choose k}\over d-1}\right\rceil$. If for some $\epsilon>0$, we have $(1-\epsilon)\rho>x^{-2}n\ln 3+2x^{-1}(1+\ln(d-1))$, then 
a.s. $\chi(\SG_{n,k}(\rho))\geq d$.
\end{alphtheorem}
Kupavskii, at the end of his paper, marked the investigation of the chromatic 
number of random Kneser hypergraphs $\KG^r_{n,k}(\rho)$ as a very interesting problem. 
In this paper, we shall study the chromatic number of random general Kneser hypergraphs  $\KG^r_{n,k}(\mathcal{H})(\rho)$.
As the first main result of this paper, we extend Theorem~\ref{Kupavskiimain} 
to the following theorem. 

\begin{theorem}\label{thm:main}
Let ${\mathcal A}=\left\{\mathcal{H}_m:\; m\in\mathbb{N}\right\}$ be a 
family of distinct hypergraphs and set 
$n=n(m)=|V(\mathcal{H}_m)|$.  Let $r=r(n)\geq 2$, $t=t(n)$, $d=d(n)$, 
and $q=q(n)$ be integer functions, where $(d-1)(t-1)+1\leq q\leq (d-1)t$ and 
let $\rho=\rho(n)\in(0,1]$ be a real function. Then we a.s. have $$\chi(\KG^r(\mathcal{H}_m)(\rho))\geq \min\left\{{|V(\mathcal{H}_m)|-{\rm alt}_r(\mathcal{H}_m,q)\over r-1},d\right\}$$ 
provided that  ${n\ln(r+1)+rt(1+\ln (d-1))}-\rho t^r\rightarrow -\infty$ as $n$ tends to infinity.
\end{theorem} 
Note that if we set $n=m$, $\mathcal{H}_n=K_n^k$, 
then $\KG^r(\mathcal{H}_n)=\KG^r_{n,k}$.
Consequently, if we set  $r=2$, $d=n-2k-2l+2$, $q={k+l\choose k}$, and  
$t=\ds\left\lceil{{k+l\choose k}\over d-1}\right\rceil$,  then the previous theorem 
results in a slightly weaker version of Kupavskii's theorem (using Kneser graphs $\KG_{n,k}$ instead of Schrijver graphs $\SG_{n,k}$).  
Also, in general, 
for $n=m$, ${\mathcal A}=\left\{K_n^k:\; n\in\mathbb{N}\right\}$, $d= \left\lceil{n-r(k+l-1)\over r-1}\right\rceil\geq 2$, and $q={k+l\choose k}$, 
Theorem~\ref{thm:main} implies that if ${n\ln(r+1)+rt(1+\ln (d-1))}-\rho t^r\rightarrow -\infty,$
then a.s.  $\chi(\KG^r_{n,k}(\rho))\geq \min\left\{{n-{\rm alt}_r(K_n^k,q)\over r-1},d\right\}.$ 
On the other hand, for the identity bijection $I:[n]\longrightarrow [n]$, since $q=\ds{k+l\choose k}$, 
we have $$\alt(K_n^k,q)\leq {\rm alt}_r(K_n^k,I,q)= r(k+l-1).$$ Consequently, we a.s. have 
$\chi(\KG^r_{n,k}(\rho))\geq \min\left\{{n-r(k+1-1)\over r-1},d\right\}= {n-r(k+1-1)\over r-1}$ 
provided that
$${n\ln(r+1)+rt(1+\ln (d-1))}-\rho t^r\rightarrow -\infty.$$  
This observation proves the next theorem provided that condition~(I) holds.   Therefore, to prove the next theorem, it suffices to consider just the second condition, which is discussed in Section~\ref{proofs}. 
\begin{theorem}\label{main:thmkneser}
Let $k=k(n)$, $r=r(n)$ and $l=l(n)$ be nonnegative integer functions 
and let $\rho=\rho(n)$ be a real function, where $2\leq r \leq {n\over k}$ and $\rho\in(0,1]$.  For $d= \left\lceil{n-r(k+l-1)\over r-1}\right\rceil\geq 2$ 
and $t=\ds\left\lceil{{k+l\choose k}\over d-1}\right\rceil$, we have a.s.
$\chi(\KG^r_{n,k}(\rho))\geq d$ provided that at least one of the followings holds;
\begin{itemize}
\item[{\rm (I)}] ${n\ln(r+1)+rt(1+\ln (d-1))}-\rho t^r\rightarrow -\infty$
\item[{\rm (II)}] $r(k+l)(\ln n+1) +rt(1+\ln (d-1))-\rho t^r\rightarrow -\infty.$
\end{itemize}
\end{theorem}

In Theorem~\ref{Kupavskiimain} and Theorem~\ref{main:thmkneser}, we deal with some quite 
complicated conditions which make this theorems difficult to use. 
To get rid  of these difficulties, Kupavskii derived some corollaries from 
Theorem~\ref{Kupavskiimain} having  simpler conditions.  
In detail, he proved that a.s. $\chi(\SG_{n,k}(\rho))\geq \chi(\KG_{n,k})-4$ 
provided that $\rho$ is fixed and $k\gg n^{3\over 4}$. 
Also, for any fixed $\rho$ and for $n-2k\ll \sqrt{n}$, he improved this lower 
bound by proving that a.s.  $\chi(\SG_{n,k}(\rho))\geq \chi(\KG_{n,k})-2$.
With a straightforward computation and by use of Theorem~\ref{main:thmkneser}, one can extend Kupavskii's results to the Kneser hypergraphs $\KG^2_{n,k}$. 

In the rest of this section, we consider 
some special cases of Theorem~\ref{main:thmkneser}, which are easy to interpret.
In this regard, we prove two corollaries (Corollary~\ref{corII} and Corollary~\ref{SGcor}), which not only extend 
Kupavskii's results  to random Kneser hypergraphs, but also improve it 
(when we deal with the case $r=2$).
\begin{corollary}\label{2main:cor}
Let $\rho\in(0,1]$ be a real number. Also, let $k=k(n)$ and $r=r(n)$ be positive integer functions, where $2\leq r \leq {n\over k}$.
If $k\gg n^{r\over 2r-1}(\ln n)^{1\over 2r-1}$, then a.s. 
$\chi(\KG^r_{n,k}(p))\geq \left\lceil{n-r(k+1)\over r-1}\right\rceil$. In particular, if $n^{\frac{r-1}{r}}\gg rn-r^2k $, then a.s. 
$\chi(\KG^r_{n,k}(p))\geq \left\lceil{n-rk\over r-1}\right\rceil$.
\end{corollary}
\begin{proof}
To prove the assertion, it suffices to check that if at least one of two conditions in 
Theorem~\ref{main:thmkneser}
holds for $l=2$ and $l=1$, respectively.
Let us first deal with the case $l=2$. We prove this case via Condition~(II) of 
Theorem~\ref{main:thmkneser}. To this end, 
we need to show that for
$d= \left\lceil{n-r(k+1)\over r-1}\right\rceil$ and 
$t=\ds\left\lceil{(k+2)(k+1)\over 2(d-1)}\right\rceil$, we have 
$r(k+2)(\ln n+1) +rt(1+\ln (d-1))-\rho t^r\rightarrow -\infty$, which clearly holds, 
since
$r(k+2)(\ln n+1)=o(t^r)$ and $rt(1+\ln d)\leq rt(1+\ln n)=o(t^r)$.

For $l=1$, note that ${n^{\frac{r-1}{r}}\over r^2}\gg d= \left\lceil{n-rk\over r-1}\right\rceil$ and for large enough $n$, we have 
$k\geq {n\over 2r}$; consequently, 
$$t=\ds\left\lceil{k+1\over (d-1)}\right\rceil\gg
{{n\over r}\over {n^{\frac{r-1}{r}}\over r^2}}=r\cdot n^{1\over r}.
$$
Now, we clearly have $n\ln(r+1)=o(t^r)$ and $rt(1+\ln d)\leq rt(1+\ln n)=o(t^r)$. 
 Using Condition~(I) of Theorem~\ref{main:thmkneser}, we have the proof completed.
\end{proof}


Next corollary is an immediate consequence of Corollary~\ref{2main:cor}.  
\begin{corollary}\label{corII}
	Let $\rho\in(0,1]$ be a real number. Also, let $k=k(n)$ and $r=r(n)$ be positive integer functions, where $2\leq r \leq {n\over k}$. Then the following assertions hold.
	\begin{itemize}
		\item[{\rm I)}] If $k\gg n^{r\over 2r-1}(\ln n)^{1\over 2r-1}$, 
		         then a.s. 
		         $$\chi(\KG^r_{n,k}(\rho))\geq\left\{	
				\begin{array}{ll}
					\chi(\KG^r_{n,k})-4 & r=2\\
					\chi(\KG^r_{n,k})-3 & r>2.
				\end{array}\right.$$
			 In particular, if  $n\not\equiv k,k+1\; (\mod r-1)$, then 
			 a.s $\chi(\KG^r_{n,k}(\rho))\geq\chi(\KG^r_{n,k})-2$.
		\item[{\rm II)}]  If  
			 $n^{\frac{r-1}{r}}\gg rn-r^2k $,  then a.s  
			$\chi(\KG^r_{n,k}(\rho))\geq\chi(\KG^r_{n,k})-2$. In particular, if
			$n\not\equiv k\; (\mod r-1)$, then a.s.  $\chi(\KG^r_{n,k}(\rho))
			\geq\chi(\KG^r_{n,k})-1$.
	\end{itemize}
\end{corollary}
Note that Kupavskii's~result (Theorem~\ref{Kupavskiimain}) provides an a.s. lower 
bound for the chromatic number of random Schrijver graphs $\SG_{n,k}(\rho)$, 
while Theorem~\ref{main:thmkneser} and Corollary~\ref{corII} concern the chromatic 
number of random Kneser  hypergraphs $\KG^r_{n,k}(\rho)$. 
The next theorem can be seen as a complementary statement for 
Theorem~\ref{Kupavskiimain}. 

\begin{theorem}\label{Kupavskiimainnew}
Let $k=k(n)\geq 2$ and $l=l(n)\geq 0$ be integer functions and $\rho=\rho(n)\in(0,1]$ 
be a real function such that  $d=n-2k-2l+2\geq 2$. Put $t=\ds\left\lceil{{k+l\choose k}\over d-1}\right\rceil$. If $r(k+l)(\ln n+1) +rt(1+\ln (d-1))-\rho t^r\rightarrow -\infty$,  then 
a.s. $\chi(\SG_{n,k}(\rho))\geq d$.
\end{theorem}

Similar to the proof of Corollary~\ref{2main:cor} and by using 
Theorem~\ref{Kupavskiimainnew} instead of 
Theorem~\ref{main:thmkneser}, we can prove the next corollary, which is an improvement of Kupavskii's result. 
\begin{corollary}\label{SGcor}
Let $\rho\in(0,1]$ be a real number and $k=k(n)\leq {n\over 2}$ be an integer function. 
If $k\gg n^{2\over 3}(\ln n)^{1\over 3}$, then a.s. 
$\chi(\SG_{n,k}(\rho))\geq \chi(\KG_{n,k})-4$. 
\end{corollary}

\subsection{\bf Coloring With No Monochromatic $K_{t,\ldots,t}^r$ Subhypergraph.}
Let $r$ and $t$ be two integers, where $r\geq 2$ and $t\geq 1$ and let $\mathcal{H}$ be a hypergraph. Also, set $K_{t,\ldots,t}^r$ to be the complete $r$-uniform $r$-partite hypergraph with $t r$ vertices such that each of its parts has $t$ vertices.
Next result concerns the minimum number of colors required in 
a coloring of $\KG^r(\mathcal{H})$ with no monochromatic $K_{t,\ldots,t}^r$ subhypergraph. For $t=1$,  any edge of $\KG^r(\mathcal{H})$ is a $K_{1,\ldots,1}^r$ subhypergraph of $\KG^r(\mathcal{H})$. Therefore, for $t=1$, any coloring of $\KG^r(\mathcal{H})$ with no monochromatic $K_{1,\ldots,1}^r$ subhypergraph is just a proper coloring of $\KG^r(\mathcal{H})$. 
Note that for $t=q=1$, $d=n$,  the next theorem implies Inequality~\ref{alihajijctb}.
\begin{theorem}\label{genalihaji}
Let $\mathcal{H}$ be a hypergraph and $\sigma:[n]\longrightarrow V(\mathcal{H})$ be an arbitrary bijection. Also, 
let $d,q,r$ and $t$ be be positive integers, where $r\geq 2$ and  $q\geq (d-1)(t-1)+1$. Then
any coloring of $\KG^r(\mathcal{H})$ with no monochromatic $K^r_{t,\ldots,t}$ uses at least 
$\min\left\{\left\lceil{n-{\rm alt}_r(\mathcal{H},\sigma,q)\over r-1}\right\rceil,d\right\}$ colors.
\end{theorem}
For a given positive integer $t$, let $l$ be the smallest nonnegative integer such that 
$$ q={k+l\choose k}\geq \left(\left\lceil{n-r(k+l-1)\over r-1}\right\rceil-1\right)\left(t-1\right)+1.$$
Theorem~\ref{genalihaji} implies that any coloring of $\KG^r_{n,k}$ with no monochromatic 
$K_{t,\ldots,t}^r$ subhypergraph uses at least 
$\left\lceil{n-r(k+l-1)\over r-1}\right\rceil$ colors. Note that the case $t=1$ concludes 
the chromatic number of Kneser hypergraphs $\KG^r_{n,k}$. \\

\noindent{\bf Plan.} This paper is organized as follows. In Section~\ref{tools}, we introduce some tools  
which will be needed throughout the paper. Section~\ref{proofs} is devoted to the 
proof of main theorems. In the last section, we present a generalization of Theorem~\ref{Kupavskiimain} with a purely combinatorial proof which implies this theorem immediately. 

\section{Tools}\label{tools}
\subsection{\bf Random General Kneser Hypergraphs}
Let $\mathcal{H}=(V(\mathcal{H}),E(\mathcal{H}))$ be a hypergraph and $r, s, C$ be positive integers, where $r,s\geq 2$. 
Also, let $\sigma:[n]\longrightarrow V(\mathcal{H})$ be a bijection.
Let $M\subseteq V(\mathcal{H})$ be an $m$-set, where 
$\sigma^{-1}(M)=\{i_1,\ldots,i_m\}$ and $i_1<\cdots<i_m$. By $\sigma_M$, we mean the following bijective map; 
$$\begin{array}{lcll}
\sigma_M: & [m] & \longrightarrow & M\\
& j &\longmapsto &\sigma(i_j).
\end{array}$$

Define $\T=\T_{\mathcal{H},C,s,\sigma}$ to be a hypergraph with vertex set $V(\mathcal{H})$ and edge set
$$E(\T)=\left\{M\subseteq V(\mathcal{H}):\; M\neq\varnothing \mbox{ and } |M|-{\rm alt}_s(\mathcal{H}[M],\sigma_M,q) > (s-1)C\right\}.$$

The next lemma, for $q=1$, is implicitly used in the proof of Theorem~\ref{alihajijctb} in~\cite{2013arXiv1302.5394A}.  Also, a similar lemma is proved in~\cite{HaMe16}.
However, for sake of completeness, we state it here with a proof.
\begin{lemma}\label{altT}
Let $\mathcal{H}=(V(\mathcal{H}),E(\mathcal{H}))$ be a hypergraph and $r, s, C$ and $q$ be positive integers, where $r,s\geq 2$. 
Then for any bijection $\sigma: [n]\longrightarrow  V(\mathcal{H})$, we have
$$\alt_r(\T,\sigma,1)\leq r(s-1)C+\alt_{rs}(\mathcal{H},\sigma,q).$$
\end{lemma}
\begin{proof}
If $\alt_r(\T,\sigma,1)=0$, then there is nothing to prove. Therefore, we may assume that 
$\alt_r(\T,\sigma,1)>0$. 

For simplicity of notation, 
without loss of generality, suppose that $V(\mathcal{H})=[n]$ and $\sigma=I$ (the identity map).
Therefore, for each $A=\{a_1,\ldots,a_m\}\subseteq [n]$ ($a_1< \cdots< a_m$), we have
$$\begin{array}{lrll}
I_A:& [m] & \longrightarrow & A\\
	     & i    &\longmapsto    & a_i.
\end{array}
$$
In view of the definition of $\alt_r(\T,I,1)$, there is an $X=(X^1,\ldots,X^r)\in (\mathbb{Z}_r\cup\{0\})^n$ 
with ${\rm alt}(X)=|X|=\alt_r(\T,I,1)$ and 
such that $E(\mathcal{H}[X^i])=\varnothing$ for each $i\in[r]$. It implies that 
$X^i\not\in E(\T)$ for each $i\in[r]$. 
Let $I_0$ be the set of all $i\in[r]$, such that 
$X^i\neq\varnothing$. Note that since $\alt_r(\T,\sigma,1)>0$, we have $I_0\neq \varnothing$.
Consequently,  for each $i\in I_0$, 
we have 
$$|X^i|-{\rm alt}_s(\mathcal{H}[X^i],I_{X^i},q)\leq (s-1)C.$$
It implies that for each $i\in I_0$,  there is at least one  
$Y_i=(Y^{i,1},\ldots,Y^{i,s})\in(\mathbb{Z}_s\cup\{0\})^{|X^i|}$ such that 
\begin{itemize}
\item $\alt(Y_i)=|Y_i| \geq |X^i|-(s-1)C$ and
\item $|E(\mathcal{H}[I_{X^i}(Y^{i,j})])|<q$ for each $j\in[s]$.
\end{itemize}
Note that for each $i\in I_0$ and each $j\in[s]$, we have 
$$Y^{ij}\subseteq \{1,\ldots,|X^i|\}\quad\mbox{ and }\quad I_{X^i}:\{1,\ldots,|X^i|\}\longrightarrow X^i.$$ 
For each $i\in [r]\setminus I_0$, set $I_{X^i}(Y^{i,1})=\cdots=I_{X^i}(Y^{i,s})=\varnothing$. 
Define 
$$
\begin{array}{lll}
Z& =& \left(I_{X^1}(Y^{1,1}),\ldots,I_{X^1}(Y^{1,s}),\ldots,
I_{X^r}(Y^{r,1}),\ldots,I_{X^r}(Y^{r,s})\right)\\
& =& (Z^1,Z^2,\ldots,Z^{rs})\in(\mathbb{Z}_{rs}\cup\{\zero\})^n.
\end{array}$$
One can simply see that $\alt(Z)=|Z|$. This implies that
$$
\begin{array}{lll}
\alt(Z) & =  & \ds\sum_{i=1}^r\sum_{j=1}^s| I_{X^i}(Y^{i,j})|\\
	  & =  & \ds\sum_{i\in I_0}|Y_i|\\ 
	  & \geq  & \ds\sum_{ i\in I_0}\left(|X^i|-(s-1)C\right)\\
	  & \geq  & |X|-r(s-1)C\\
	  & =  & \alt_r(\T,\sigma,1)-r(s-1)C.
\end{array}$$ 
In view of the definition of $\alt_{rs}(\mathcal{H},I,q)$ and since $|E(\mathcal{H}[Z^l])|<q$  for each $l\in [rs]$, 
we have 
$$\alt_{rs}(\mathcal{H},I,q)\geq \alt(Z)\geq\alt_r(\T,I,1)-r(s-1)C,$$
as desired.
\end{proof}
Now, we are ready to state the main lemma, which has a key role in the paper.
For the proof of this lemma, we need the following version of $Z_p$-Tucker lemma.
\begin{alphlemma}{\rm ($\Z_p$-Tucker lemma~\cite{MR2793613,MR1893009})}\label{zptucker}
Let $m, n, p$, and $\alpha$ be nonnegative integers, where $m, n\geq 1$, $m\geq \alpha \geq 0$, and $p$
is prime. Let
$$
\begin{array}{rl}
  \lambda:\ (\mathbb{Z}_p\cup\{0\})^n\setminus\{\zero\} & \longrightarrow \mathbb{Z}_p\times[m] \\
  X & \longmapsto(\lambda_1(X),\lambda_2(X))
\end{array}
$$
be a map satisfying the following properties:
\begin{itemize}
\item[{\rm (i)}] $\lambda$ is a $\mathbb{Z}_p$-equivariant map, that is, for each $\varepsilon\in \mathbb{Z}_p$,
we have $\lambda(\varepsilon X)=(\varepsilon\lambda_1(X),\lambda_2(X))$,
\item[{\rm (ii)}] for $X_1\subseteq X_2\in (\mathbb{Z}_p\cup\{0\})^n\setminus\{\zero\}$,
      if $\lambda_2(X_1)=\lambda_2(X_2)\leq\alpha$, then $\lambda_1(X_1)=\lambda_1(X_2)$,
\item[{\rm (iii)}] for $X_1\subseteq\cdots \subseteq X_p\in (\Z_p\cup\{0\})^n\setminus\{\zero\}$,
      if $\lambda_2(X_1)=\cdots=\lambda_2(X_p)\geq\alpha+1$,
      then 
$$\left|\left\{\lambda_1(X_1),\ldots,\lambda_1(X_p)\right\}\right|<p.$$
\end{itemize}
Then $\alpha+(m-\alpha)(p-1) \geq n$.
\end{alphlemma}

\begin{lemma}\label{lem:main}
Let $d,q,r$, and $t$ be positive integers, where $d,r\geq 2$, and  $q\geq (d-1)(t-1)+1$.
Let $\mathcal{H}$ be a hypergraph and $\preceq$ be a total ordering on the power set 
of $V(\mathcal{H})$, which refines the partial ordering according to size.
Moreover, let $\sigma: [n]\longrightarrow V(\mathcal{H})$ be a bijection.
Then for any coloring  $c:E(\mathcal{H})\longrightarrow [C]$, where 
$1\leq C<\min\{{n-{\rm alt}_r(\mathcal{H},\sigma,q)\over r-1},d\}$, there 
exists an $r$-tuple $(N_1,\ldots,N_r)$ with the following properties; 
\begin{itemize}
\item $N_1,\ldots,N_r$ are pairwise disjoint subsets of $[n]$.
\item For each $j\in[r]$, $|E(\mathcal{H}[\sigma(N_j)])|\geq q$.
\item For each $j\in[r]$, there are $t$ distinct edges $e_{1,j},\ldots,e_{t,j}\subseteq \sigma(N_j)$ chosen from the last $q$ largest 
edges in  $E(\mathcal{H}[\sigma(N_j)])$ {\rm (}according to the total ordering $\preceq${\rm )} such that all edges in $\{e_{i,j}:i\in[t]\;\&\; j\in[r]\}$ receive the same color $c(N_1,\ldots,N_r)\in[C]$.
\end{itemize}
\end{lemma}
\begin{proof}
The proof is divided into two parts. 
First, we prove the theorem when $r$ is prime. 
Then, we reduce the nonprime case to the prime case, which completes the proof.

First, assume that $r=p$ is a prime number. Consider an arbitrary coloring $c:E(\mathcal{H})\longrightarrow [C]$ such that $1\leq C<\min\{{n-{\rm alt}_r(\mathcal{H},\sigma,q)\over r-1},d\}$. 
Without loss of generality and for simplicity of notation,  we may assume that $V(\mathcal{H})=[n]$ and  $\sigma=I$ is the identity map.
Set  $m={\rm alt}_p({\mathcal H},I,q)+C$
and $\alpha={\rm alt}_p({\mathcal H},I,q)$.
Let  $$\begin{array}{crcc}
\lambda: & (\mathbb{Z}_p\cup\{0\})^n\setminus\{\zero\} &\longrightarrow & \mathbb{Z}_p\times[m]\\
& X &\longmapsto & (\lambda_1(X),\lambda_2(X))
\end{array}$$ 
be a map defining as follows. 
\begin{itemize}
\item If ${\rm alt}(X)\leq \alpha$, then define $\lambda_1(X)$ to be the first 
nonzero coordinate of $X$ and $\lambda_2(X)={\rm alt}(X)$.
\item If ${\rm alt}(X)\geq \alpha+1$, then, in view of the definition of 
${\rm alt}_p({\mathcal H},I,q)$, there is at least one $i\in[p]$ such that $|E(\mathcal{H}[X^i])|\geq q$. 
Choose $\omega^i\in\mathbb{Z}_p$ such that 
$$X^i=\max\{X^j:\; |E(\mathcal{H}[X^j])|\geq q\},$$
where the maximum is taken according to the total ordering $\preceq$.
Now, see all edges in $E(\mathcal{H}[X^i])$ as a chain (according to the total ordering $\preceq$) and 
consider the last $q$ edges of this chain.
In other words, if $E(\mathcal{H}[X^i])=\{e_1,\ldots,e_{m}\}$, where $e_1\prec\cdots\prec e_m$, then
consider $e_{m-q+1},\ldots,e_{m}$.
Define $c(X)$ to be the 
most popular color amongst all colors assigned  to these $q$ edges.
If there is more than one such a color, then choose the maximum one.  
Clearly the frequency of this color is at least $\lceil{q\over C}\rceil\geq t$
(note that ${q\over C}\geq{q\over d-1 }> t-1$).
Define $$\lambda(X)=(\omega^i,\alpha+c(X)).$$
\end{itemize}
It is straightforward to check that the map  $\lambda$ satisfies  Property~(i) and Property~(ii) of Lemma~\ref{zptucker}. Since 
$$n-{\rm alt}_p({\mathcal H},I,q)=n-\alpha> (m-\alpha)(p-1)=C(p-1),$$
the map $\lambda$ does~not satisfy Property~(iii) of Lemma~\ref{zptucker}. 
Thus, there is a chain $X_1\subseteq\cdots \subseteq X_p\in (\mathbb{Z}\cup\{0\})^n\setminus\{\zero\}$,
      such that  $i=\lambda_2(X_1)=\cdots=\lambda_2(X_p)\geq\alpha+1$ and  
$\left|\left\{\lambda_1(X_1),\ldots,\lambda_1(X_p)\right\}\right|=p.$ 
Hence, we have $\left\{\lambda_1(X_1),\ldots,\lambda_1(X_p)\right\}=\Z_p$.
Let $\pi:[p]\longrightarrow [p]$ be the bijection for which we have 
$\lambda_1(X_j)=\omega^{\pi(j)}$ for each $j\in[p]$.  
Define $N_j=X_j^{\pi(j)}\subseteq X_p^{\pi(j)}$
for each $j\in[p]$. Since  
the sets $X_p^j$'s are pairwise disjoint, the sets $N_j$'s are pairwise disjoint as well. 
In view of the definition of $\lambda$, for each $j\in[p]$, there are at least $t$ 
edges $e_{1,j},\ldots,e_{t,j}\subseteq N_j$ such that 
these edges are amongst the last $q$ largest edges in $E(\mathcal{H}[N_j])$ and 
$c(e_{1,j})=\cdots=c(e_{t,j})=c(X_j)=i-\alpha$. It implies that all edges in $\{e_{i,j}:i\in[t]\;\&\; j\in[r]\}$ receive the same color $i-\alpha$. Clearly, for $c(N_1,\ldots,N_p)=i-\alpha$, the $p$-tuple $(N_1,\ldots,N_p)$ has the desired properties. 
\begin{lemma}\label{reduction}
If Lemma~{\rm\ref{lem:main}} holds for $r=r_1$ and $r=r_2$, then it holds for $r=r_1r_2$.
\end{lemma}
\begin{proof}
Let $c:E(\mathcal{H})\longrightarrow [C]$ be a coloring such that 
$$1\leq C<\min\left\{{|V(\mathcal{H})|-{\rm alt}_{r_1r_2}(\mathcal{H},\sigma,q)\over r_1r_2-1},d\right\}.$$
Note that this implies that $C<d$. 
Consider the hypergraph $\T=\T_{\mathcal{H},C,r_2,\sigma}$. First,
we define a coloring $f:E(\T)\longrightarrow [C]$.
For each $M\in E(\T)$, in view of the definition of $\T$, we have 
$$|M|-{\rm alt}_{r_2}(\mathcal{H}[M],\sigma_M,q) > (r_2-1)C.$$
Hence, 
$$1\leq C<\min\left\{{|M|-{\rm alt}_{r_2}(\mathcal{H}[M],\sigma_M,q)\over r_2-1},d\right\}.$$
Consider the hypergraph $\mathcal{H}[M]$ and the coloring $c$ restricted to the edges of $\mathcal{H}[M]$. 
Let $(N_1,\ldots,N_{r_2})$ be an 
$r_2$-tuple whose existence is ensured since we have assumed that Lemma~\ref{lem:main} is true for $r=r_2$.
Note that $N_1,\ldots,N_{r_2}$ are pairwise disjoint subsets of $\{1,\ldots,|M|\}$.
Now, define $f(N)=c(N_1,\ldots,N_{r_2})$.
In view of lemma~\ref{altT}, we have
$$\begin{array}{lll}
n-\alt_{r_1}(\T,\sigma,1) & \geq & n-r_1(r_2-1)C-\alt_{r_1r_2}(\mathcal{H},\sigma,q)\\
			       &    >   & (r_1r_2-1)C-r_1(r_2-1)C\\
			       &   =    & (r_1-1)C.
\end{array}$$
It implies that 
$$1\leq C<\min\left\{{|V(\T)|-{\rm alt}_{r_1}(\T,\sigma,1)\over r_1-1},d\right\}.$$
Since Lemma~\ref{lem:main} holds for $r=r_1$, if we set $t=q=1$, then
there are $r_1$ pairwise vertex-disjoint edges $M_1,\ldots,M_{r_1}\in E(\T)$
such that $f(M_1)=\cdots=f(M_{r_1})=i$. 
Now, for each $i\in[r_1]$, let $(N_{i,1},\ldots,N_{i,r_2})$ be the $r_2$-tuple, which is
used for the definition of  $f(M_i)$. Now, one can see that the $r_1r_2$ tuple
$$P=\left(N_{1,1},\ldots,N_{1,r_2},\ldots,N_{r_1,1},\ldots,N_{r_1,r_2}\right)$$  with $c(P)=i$ has the desired properties. 
\end{proof}
By induction, Lemma~\ref{reduction},  and the fact  that
Lemma~\ref{lem:main} is true for any prime number $r$, the proof is 
completed.
\end{proof}

\subsection{\bf Random Kneser Hypergraphs and Schrijver Graphs}
In this subsection, we present two specializations of Lemma~\ref{lem:main}, which will be useful for computing the chromatic number of random Kneser hypergraphs $\KG^r_{n,k}(\rho)$ and random Schrijver Graphs $\SG_{n,k}(\rho)$.  
\begin{lemma}\label{lem:mainken}
Let $n, k, r$ and $l$ be nonnegative integers, where $r\geq 2$, $k\geq 1$,  $n\geq r k$, and 
$d= \left\lceil{n-r(k+l-1)\over r-1}\right\rceil\geq 2$.  
Set $t=\ds\left\lceil{{k+l\choose k}\over d-1}\right\rceil$. 
Then for any coloring  $c:{[n]\choose k}\longrightarrow [C]$, where 
$1\leq C<d$, there 
exists an $r$-tuple $(N_1,\ldots,N_r)$ with the following properties; 
\begin{itemize}
\item $N_1,\ldots,N_r$ are pairwise disjoint $(k+l)$-subsets of $[n]$.
\item For each $j\in[r]$, there are $t$ distinct $k$-subsets 
$e_{1,j},\ldots,e_{t,j}\subseteq N_j$  such that all members of  
$\{e_{i,j}:i\in[t]\;\&\; j\in[r]\}$ receive the same color $c(N_1,\ldots,N_r)\in[C]$.
\end{itemize}
\end{lemma} 
\begin{proof}
Consider an arbitrary coloring $c:{[n]\choose k}\longrightarrow [C]$, where 
$1\leq C<d=\left\lceil{n-r(k+l-1)\over r-1}\right\rceil$. 
Clearly, the assumption $d\geq 2$ implies that $n\geq r(k+l)$. 
Define the coloring $f:V(\KG^r_{n,k+l})\longrightarrow [C]$ as follows.  
For each $(k+l)$-set $L\in V(\KG^r_{n,k+l})$, set $f(L)$ to be the most popular
color (with respect to the coloring $c$) amongst the members of 
$\{A\;:\; |A|=k\mbox{ and } A\subseteq L\}\subseteq {[n]\choose k}$.  
If there is more than one such a color, then choose the maximum one. 
We already know that $\chi(\KG^r_{n,k+l})=d$.
Since $C< d=\chi(\KG^r_{n,k+l})$, the coloring $f$ is~not proper.
Consequently, there are $r$ pairwise disjoint $(k+l)$-sets $N_1,\ldots,N_r\subseteq [n]$ 
such that $f(N_1)=\cdots=f(N_r)=i\in [C]$. 
In view of the definition of $f$, one can simply see that
$P=(N_1,\ldots,N_r)$ with  $c(P)=i$ is the desired $r$-tuple.
\end{proof}

Also, we can have a similar statement for Schrijver graphs.
\begin{lemma}\label{SGlemma}
Let $n, k$ and $l$ be nonnegative integers, where  $n\geq 2 k\geq 2$, and 
$d= n-2(k+l-1)\geq 2$.  Set $t=\left\lceil{{k+l\choose k}\over d-1}\right\rceil$.
For any coloring  $c:{[n]\choose k}_{stable}\longrightarrow [C]$ with 
$1\leq C<d$, there 
is a pair $(N_1,N_2)$ with the following properties.
\begin{itemize}
\item $N_1$ and $N_2$ are disjoint stable $(k+l)$-subsets of $[n]$. 
\item For $j=1,2$, there are $t$ distinct stable $k$-sets 
$e_{1,j},\ldots,e_{t,j}\subseteq N_j$ such 
that all members of $\{e_{i,j}:i\in[t]\;\&\; j\in[2]\}$ receive the same color $c(N_1,N_2)\in[C]$.
\end{itemize}
\end{lemma}
\begin{proof}
Consider an arbitrary coloring $c:{[n]\choose k}_{stable}\longrightarrow [C]$, where 
$1\leq C<d=n-2(k+l-1)$. 
Define the coloring $f:V(\SG_{n,k+l})\longrightarrow [C]$ as follows. 
For each stable $(k+l)$-set $L\in V(\SG_{n,k+l})$, set $f(L)$ to be the  
most popular color (with respect to the coloring $c$) amongst the members 
of $\{A\;:\; |A|=k \mbox{ and } A\subseteq L\}\subseteq {[n]\choose k}_{stable}$. 
If there is more than one such a color, then choose the maximum one. 
Since $C<\chi(\SG_{n,k+l})=d$, there are two disjoint stable $(k+l)$-sets $N_1,N_2\subseteq [n]$ such that $f(N_1)=f(N_2)$, which completes the proof.
\end{proof}

\section{Proofs of Main Results}\label{proofs}
This section is completely devoted to the proof of main results. \\

\noindent{\bf Proof of Theorem~\ref{thm:main}.}
Assume that at least one of two mentioned conditions in the assertion of the theorem holds.
For an arbitrary $m\in\mathbb{N}$, set $\mathcal{H}_m=\mathcal{H}$.
Let $\preceq$ be a total ordering on the power set 
of $V(\mathcal{H})$, which refines the partial ordering according to size. Let 
$\sigma:[n]\longrightarrow V(\mathcal{H})$ be a bijection for which we have
$\alt_r(\mathcal{H},\sigma,q)=\alt_r(\mathcal{H},q)$.\\

\noindent{\bf The Event $\mathbf{E}$.}  Define $\mathbf{E}$ to be the event that $\KG^r(\mathcal{H})(\rho)$ has some proper $C$-coloring for some $1\leq C<\min\left\{{|V(\mathcal{H})|-{\rm alt}_r(\mathcal{H},\sigma,q)\over r-1},d\right\}.$
Clearly, to complete the proof, it suffices to show that $Pr(\mathbf{E})\rightarrow 0$ as $m\rightarrow +\infty$.  

Consider an arbitrary $P=(M_1,\ldots,M_r)$ 
such that $M_1,\ldots, M_r$ 
are pairwise disjoint subsets of $[n]$ and 
$|E(\sigma(M_i))|\geq q$ for each $i\in [r]$. Now, for each $i\in[r]$, 
see all edges in $E(\mathcal{H}[\sigma(M_i)])$ as a chain 
according to the total ordering $\preceq$ and 
consider the last $q$ edges appearing in this chain. 
Let $U_i=U_i(P)$ be set of those edges. \\

\noindent{\bf The Event $\mathbf{A}$.} 
Define $\mathbf{A}(P)$ to be the event that for each $i\in [r]$, there is
a $t$-subset $V_i\subseteq U_i$ such that 
the subhypergraph $\KG^r(\mathcal{H})(\rho)[V_1,\ldots,V_r]$
has no edge. 
Now, define the event $\mathbf{A}$ to be the union of all $\mathbf{A}(P)$'s, i.e., 
 $$\mathbf{A}=\bigcup \mathbf{A}(P),$$
where the union is taken over all $P=(M_1,\ldots,M_r)$  
such that $M_1,\ldots,M_r$ 
are pairwise disjoint subsets of $[n]$ and
$|E(\sigma(M_i))|\geq q$ 
for each $i\in [r]$.\\

Let $c:E(\mathcal{H})\longrightarrow [C]$ be a proper coloring for $\KG^r(\mathcal{H})(\rho)$,
where $C<\min\left\{{|V(\mathcal{H})|-{\rm alt}_r(\mathcal{H},\sigma,q)\over r-1},d\right\}.$ 
Consider the $r$-tuple  $P=(N_1,\ldots,N_r)$ whose existence is 
ensured by Lemma~\ref{lem:main}. Without loss of generality, we may assume that 
$c(N_1,\ldots,N_r)=1$.
Now, for each $j\in[r]$, 
see all edges in $E(\mathcal{H}[\sigma(N_j)])$ as a chain 
according to the total ordering $\preceq$ and 
consider the last $q$ edges appearing in this chain. 
Let $U_j=\{f_1^j,\ldots,f_q^j\}$ be set of those edges.  
For each $j\in[r]$, let 
$V_{j}\subseteq U_j$ be the set of edges receiving color $1$.  
Clearly,  the subhypergraph 
$\KG^r(\mathcal{H})[V_1,\ldots,V_r]$ is a monochromatic subhypergraph of 
$\KG^r(\mathcal{H})$. Therefore, since $c$ is a proper coloring,  $\KG^r(\mathcal{H})(\rho)[V_1,\ldots,V_r]$  has no edge, which implies that $\mathbf{A}(P)\subseteq \mathbf{A}$ is happened. Hence, we have $\mathbf{E}\subseteq \mathbf{A}$.
 
Also, note that since $\mathcal{H}_m$'s are distinct, if $m\rightarrow +\infty$, then 
$n=n(m)\rightarrow +\infty$.
Consequently, if we prove that $Pr(\mathbf{A})\rightarrow 0$  as $n\rightarrow +\infty$, then we have
$Pr(\mathbf{E})\rightarrow 0$ as $m\rightarrow +\infty$, as desired. 
First, note that  
$$\begin{array}{lll}
Pr(A) &\leq & \ds\sum {q\choose t}^r(1-\rho)^{t^r}\\
	 & \leq & \left(r+1\right)^n\left({eq\over t}\right)^{rt}e^{-\rho t^r}\\ 
	 & \leq & e^{-\rho t^r+n\ln(r+1)+rt(1+\ln (d-1))},
\end{array}$$ 
where the summation is taken over all $P=(M_1,\ldots,M_r)$ such that $M_1,\ldots,M_r$ 
are pairwise disjoint subsets of $[n]$ and 
$|E(\mathcal{H}[\sigma(M_i)])|\geq q$ for each $i\in [r]$ (Note that the number 
of such $P$'s is at most $(r+1)^n$).
Thus, if 
$${n\ln(r+1)+rt(1+\ln (d-1))}-\rho t^r\rightarrow -\infty,$$ 
then $Pr(\mathbf{A})\rightarrow 0$. This completes the proof. \hfill$\square$\\

The next proof is almost the same as the prior proof. However, for the ease of reading, 
we state it here completely. \\  

\noindent{\bf Proof of Theorem~\ref{main:thmkneser}.}
In view of the discussion before the statement of Theorem~\ref{main:thmkneser},
it is enough to consider that Condition~(II) holds. 
For random Kneser hypergraph $\KG^r_{n,k}(\rho)$, similar to the proof of Lemma~\ref{thm:main},  we shall introduce two events $\mathbf{E}_n$ and $\mathbf{A}_n$.\\
\noindent{\bf The Event $\mathbf{E}_n$.}  Define $\mathbf{E}_n$ to be the 
event that $\KG^r_{n,k}(\rho)$ has some proper $C$-coloring for some $1\leq C<d.$ 
Clearly, to complete the proof, it suffices to show that 
$Pr(\mathbf{E}_n)\rightarrow 0$ as $n\rightarrow +\infty$.  
Consider an arbitrary $P=(M_1,\ldots,M_r)\in(\Z_r\cup\{0\})^n$ 
such that $M_1,\ldots,M_r$ are pairwise disjoint $(k+l)$-subsets of $[n]$.
Let $$U_i=U_i(P)=\{A\;:\; |A|=k,\; A\subseteq M_i\}
\subseteq V(\KG^r_{n,k}).$$ 

\noindent{\bf The Event $\mathbf{A}$.} 
Define $\mathbf{A}_n(P)$ to be the event that for each $i\in [2]$, there is
a $t$-subset $V_i\subseteq U_i$ such that 
the subhypergraph $\KG^r_{n,k}(\rho)[V_1,\ldots,V_r]$
has no edge. 
Now, define the event $\mathbf{A}_n$ to be the union of all $\mathbf{A}_n(P)$'s, i.e., 
 $$\mathbf{A}_n=\bigcup \mathbf{A}_n(P),$$
where the union is taken over all $P=(M_1,\ldots,M_r)\in (\Z_r\cup\{0\})^n$ 
such that $M_1,\ldots,M_r$ are pairwise disjoint  $(k+l)$-subsets of $[n]$.   

Let $c:{[n]\choose k}\longrightarrow [C]$ be a proper coloring for $\KG^r_{n,k}(\rho)$,
where $C<d.$ 
Consider the $r$-tuple  $P=(N_1,\ldots,N_r)$ whose existence is 
ensured by Lemma~\ref{lem:mainken}. Without loss of generality, we may assume that 
$c(N_1,\ldots,N_r)=1$. Let $$U_i=\{A\;:\; |A|=k,\; A\subseteq N_i\}=\{f_1^j,\ldots,f_q^j\},$$
where $q={k+l\choose k}$. 
For each $j\in[r]$, let 
$V_{j}\subseteq U_j$ be the set of edges receiving color $1$.  
Clearly,  the subhypergraph 
$\KG^r_{n,k}[V_1,\ldots,V_r]$ is a monochromatic subhypergraph of 
$\KG^r_{n,k}$. Therefore, since $c$ is a proper coloring,  $\KG^r_{n,k}(\rho)[V_1,\ldots,V_r]$  has no edge, which implies that $\mathbf{A}_n(P)\subseteq \mathbf{A}_n$ is happened. Hence, we have $\mathbf{E}_n\subseteq \mathbf{A}_n$.
 
Consequently, if we prove that $Pr(\mathbf{A}_n)\rightarrow 0$  as $n\rightarrow +\infty$, then we have
$Pr(\mathbf{E}_n)\rightarrow 0$ as $n\rightarrow +\infty$, as desired. 
First, note that  
$$\begin{array}{lll}
Pr(A) &\leq & \ds\sum {q\choose t}^r(1-\rho)^{t^r}\\ \\
	 & \leq & {n\choose k+l}^r\left({eq\over t}\right)^{rt}e^{-\rho t^r}\\ \\
	 & \leq & ({ne\over k+l})^{r(k+l)}\left({eq\over t}\right)^{rt}e^{-\rho t^r}\\ \\
	 & \leq & e^{-\rho t^r+r(k+l)(\ln n+1) +rt(1+\ln (d-1))},
\end{array}$$ 
where the summation is taken over all $P=(M_1,\ldots,M_r)$ such that $M_1,\ldots,M_r$ are pairwise disjoint $(k+l)$-subsets of $[n]$. (Note that the number 
of such $P$'s is at most ${n\choose k+l}^r$).
Thus, if 
$$r(k+l)(\ln n+1) +rt(1+\ln (d-1))-\rho t^r\rightarrow -\infty,$$
then $Pr(\mathbf{A}_n)\rightarrow 0$. This completes the proof.\hfill$\square$
\\

\noindent{\bf Proof of Theorem~\ref{Kupavskiimainnew}.}
If we set $r=2$ and use Lemma~\ref{SGlemma} instead of Lemma~\ref{lem:mainken}, the proof follows by almost ``copy-pasting'' the proof of Theorem~\ref{main:thmkneser}.
\hfill$\square$\\

\noindent{\bf Proof of Theorem~\ref{genalihaji}.}
Let $c:V(\KG^r(\mathcal{H}))\longrightarrow [C]$ be a coloring such that 
$\KG^r(\mathcal{H})$ has no monochromatic $K_{t,\ldots,t}^r$ subhypergraph, where  
$1\leq C<\min\left\{{n-{\rm alt}_r(\mathcal{H},\sigma,q)\over r-1},d\right\}.$
Let $e_{i,j}$'s be the edges whose existence is ensured by Lemma~\ref{lem:main}.
If we set $W_j=\{e_{1,j},\ldots,e_{t,j}\}$, then the subhypergraph 
$\KG^r(\mathcal{H})[W_1,\ldots,W_r]$ is a monochromatic $K^r_{t,\ldots,t}$ 
subhypergraph of $\KG^r(\mathcal{H})$, a contradiction.\hfill$\square$\\

\section{Another Extension of Theorem~\ref{Kupavskiimain}\label{SG}}
Note that Kupavskii's result (Theorem~\ref{Kupavskiimain}) concerns the 
chromatic number of Schrijver graphs, while if we use Theorem~\ref{thm:main} 
(set $r=2$ and $\HH_m=([n],{[n]\choose k}_{stable})$)  to obtain a lower bound 
for the chromatic number of random Schrijver graphs, then it implies 
a lower bound which is~not as well as the lower bound stated in 
Theorem~\ref{Kupavskiimain}. Actually, one can see that this lower bound  
is the lower bound stated in Theorem~\ref{Kupavskiimain} minus one. 
Motivated by the this discussion, 
in this section, we present another extension of Theorem~\ref{thm:main}, 
which immediately  implies Theorem~\ref{Kupavskiimain}.  
Actually, with some slightly modifications in the proof of Lemma~\ref{lem:main}, 
we can have a similar statement which is helpful for the case of Schrijver graphs. 

Set $\Z_2=\{+,-\}$. For each $X=(x_1,\ldots,x_n)\in\{+,-,0\}^n$, define
$$X^+=\{i\in[n]:\; x_i=+\}\quad\mbox{and}\quad X^-=\{i\in[n]:\; x_i=-\}.$$
Let  $\mathcal{H}$ be a hypergraph and $\sigma: [n]\longrightarrow V(\mathcal{H})$ be a bijection.
Define 
$$\salt(\mathcal{H},\sigma,q)=\max\left\{{\rm alt}(X):\; X\in\{+,-,0\}^n\;\mbox{ s.t. }  \ds\min_{\varepsilon\in\{+,-\}}(|E(\mathcal{H}[\sigma(X^\varepsilon)])|)\leq q-1\right\}.$$
Now, set 
$$\salt(\mathcal{H},q)=\min_\sigma \salt(\mathcal{H},\sigma,q),$$
where the minimum is taken over all bijection $\sigma:[n]\longrightarrow V(\mathcal{H}).$
For $q=1$, we prefer to use $\salt(\mathcal{H})$ instead of $\salt(\mathcal{H},1)$.
The present authors~\cite{2013arXiv1302.5394A} proved that 
$n-\salt(\mathcal{H})+1$ is a lower bound for the chromatic number of $\KG(\mathcal{H})$. 
One can simply see that $\salt\left([n],{[n]\choose k}_{stable}\right)=2k-1$. 
Hence, in view of last mentioned lower bound, we have an exact 
lower bound for the chromatic number of Schrijver graphs $\SG_{n,k}$. Also, 
as it is expected, we can have the following lemma which is similar to Lemma~\ref{lem:main}. 
\begin{lemma}\label{main:lemmaSG}
Let $d,q$ and $t$ be a positive integers, where $d\geq 2$ and  $q\geq (d-1)(t-1)+1$.
Let $\mathcal{H}$ be a hypergraph and $\preceq$ be a total ordering on the power set 
of $V(\mathcal{H})$, which refines the partial ordering according to size.
Moreover, let $\sigma: [n]\longrightarrow V(\mathcal{H})$ be a bijection.
Then for any coloring  $c:E(\mathcal{H})\longrightarrow [C]$, where 
$C<\min\{n-{\rm salt}(\mathcal{H},\sigma,q)+1,d\}$, there 
is an ordered pair $(N_1,N_2)$ with the following properties.
\begin{itemize}
\item $N_1,N_2$ are pairwise disjoint subsets of $[n]$.
\item For $j=1,2$, $|E(\mathcal{H}[\sigma(N_j)])|\geq q$.
\item For $j=1,2$, there are $t$ distinct edges $e_{1,j},\ldots,e_{t,j}\subseteq \sigma(N_j)$ chosen from the last $q$ largest 
edges in  $\sigma(N_j)$ {\rm (}according to the total ordering $\preceq${\rm )} such 
that all edges in $\{e_{i,j}:i\in[t]\;\&\; j\in[2]\}$ receive the same color $c(N_1,N_2)$.
\end{itemize}
\end{lemma}
\noindent{\bf Sketch of Proof.} 
Consider an arbitrary coloring $c:E(\mathcal{H})\longrightarrow [C]$ such that $1\leq C<\min\{n-\salt(\mathcal{H},\sigma,q)+1,d\}$. 
Without loss of generality and for simplicity of notation,  we may assume that $V(\mathcal{H})=[n]$ and  $\sigma=I$ is the identity map.
In view of the definition of $\salt(\mathcal{H},I,q)$, 
for any $X\in\{+,-,0\}^n\setminus\{\zero\}$ with $\alt(X)\geq \salt(\mathcal{H},I,q)+1$, 
we have $|E(\mathcal{H}[X^\varepsilon])|\geq q$ for each $\varepsilon\in\{+,-\}$. See all edges in $E(\mathcal{H}[X^\varepsilon])$ as a chain (according to the total ordering $\preceq$) and 
consider the last $q$ edges of this chain and 
define $g(X^\varepsilon)$  to be the maximum  
most popular color amongst colors assigned  to these $q$ edges.
Now, set $g(X)=\max\left(g(X^+),g(X^-)\right)$. Note that if there is an 
$X\in\{+,-,0\}^n\setminus\{\zero\}$ with $\alt(X)\geq \salt(\mathcal{H},\sigma,q)+1$ 
and such that $g(X^+)=g(X^-)=i$, then 
for $N_1=X^+$ and $N_2=X^-$, 
one can simply see that the pair $(N_1,N_2)$ 
with $c(N_1,N_2)=i$ has the desired properties.
Therefore, we may assume that  $g(X^+)\neq g(X^-)$ 
for each $X\in\{+,-,0\}^n\setminus\{\zero\}$ with $\alt(X)\geq \salt(\mathcal{H},I,q)+1$. 
Note that it implies that $g(X)\geq 2$ for each $X\in\{+,-,0\}^n\setminus\{\zero\}$ with $\alt(X)\geq \salt(\mathcal{H},I,q)+1$. 
Set $p=2$, 
$\alpha=\salt(\mathcal{H},I,q)$ and $m=\salt(\mathcal{H},I,q)+C-1$. 
Now, we are ready to define a map
$\lambda: \{+,-,0\}^n\setminus\{\zero\}\longrightarrow \{+,-\}\times [m]$.
Consider an arbitrary $X\in\{+,-,0\}^n\setminus\{\zero\}$. 
If $\alt(X)\leq \alpha$, then define $\lambda(X)=(\varepsilon, \alt(X))$,
where $\varepsilon$ is the first nonzero coordinate of $X$.  
If $\alt(X)\geq \alpha+1$, then set $\lambda(X)=(\varepsilon,\alpha+g(X)-1)$, 
where $\varepsilon$ is $+$ if $g(X^+)>g(X^-)$ and is $-$ otherwise. By the same approach as in
the proof of Lemma~\ref{lem:main}, the proof follows with no difficulty.  \hfill$\square$

With the same approach as we used to derive Theorem~\ref{thm:main} from Lemma~\ref{lem:main}, we can prove the following theorem from Lemma~\ref{main:lemmaSG}.
\begin{theorem}\label{sgtheorem}
	Let ${\mathcal A}=\left\{\mathcal{H}_m:\; m\in\mathbb{N}\right\}$ be a 
	family of distinct hypergraphs and set
	$n=n(m)=|V(\mathcal{H}_m)|$.  Also, let $t=t(n)$, $d=d(n)$, 
	and $q=q(n)$ be integer functions, where
	$d\geq 2$ and $(d-1)(t-1)+1\leq q\leq (d-1)t$, and 
	 let $\rho=\rho(n)\in(0,1]$ be a real function. 
	Then we a.s. have $$\chi(\KG(\mathcal{H}_m)(\rho))\geq \min\left\{|V(\mathcal{H}_m)|-
	\salt(\mathcal{H}_m,q)+1,d\right\}$$
	provided that ${n\ln3+2t(1+\ln(d-1))}-\rho t^2\rightarrow -\infty$ as $n$ tends to infinity.
\end{theorem}
Since $\KG\left([n],{[n]\choose k}_{stable}\right)=\SG_{n,k}$  
and $\salt\KG\left([n],{[n]\choose k}_{stable}\right)=2k-1$, Theorem~\ref{sgtheorem} implies  
Theorem~\ref{Kupavskiimain}. Hence, it is a generalization of 
Theorem~\ref{Kupavskiimain} with a combinatorial proof. 
Also, similar to the proof of Theorem~\ref{main:thmkneser}, 
we have a purely combinatorial proof for the Kupavskii's~heorem (Theorem~\ref{Kupavskiimain}). \\

It might be intriguing that we state Theorem~\ref{sgtheorem} just in the case of graphs while it seems that these results remain true even for hypergraphs. 
Actually, For a hypergraph $\mathcal{H}$, we can naturally generalize 
$\salt(\mathcal{H})$ to  $\salt_r(\mathcal{H})$.  However, for any hypergraph $\mathcal{H}$, the value of $\salt_r(\mathcal{H})$ is equal to $|V(\mathcal{H})|$, which clearly makes this generalization useless.  

In the proof of Theorem~\ref{genalihaji}, if we set $r=2$ and use Lemma~\ref{main:lemmaSG} instead of Lemma~\ref{lem:main}, then we have the following theorem. It should be mentioned that this result is already proved in~\cite{2013arXiv1302.5394A} for $t=1$.
\begin{theorem}
Let $\mathcal{H}$ be a hypergraph and $\sigma:[n]\longrightarrow V(\mathcal{H})$ be an arbitrary bijection. Also, 
let $d,q$ and $t$ be be positive integers, where  $q\geq (d-1)(t-1)+1$. Then 
any coloring of $\KG(\mathcal{H})$ with no monochromatic $K_{t,t}$ uses at least 
$\min\left\{n-{\rm salt}(\mathcal{H},\sigma,q)+1,d\right\}$ colors.
\end{theorem}

\subsection*{Acknowledgments} 
The authors would like to thank Dr. Andrey~Kupavskii for finding a computational problem in the previous version of the paper. 
The research of Hossein Hajiabolhassan was in part supported by a grant from IPM (No. 94050128).  
\def\cprime{$'$} \def\cprime{$'$}

\end{document}